\documentclass[a4paper,11pt,reqno]{amsart}
\usepackage{amsmath,amsthm,amssymb,xcolor,mathtools,stmaryrd,diagbox,mathrsfs,bbm,euscript}
\usepackage[left=2.7cm,right=2.7cm,top=3.5cm,bottom=2.75cm]{geometry}
\usepackage[colorlinks=true, linktocpage=true, linkcolor=red!70!black, citecolor=green!50!black, urlcolor=black]{hyperref}
\allowdisplaybreaks
\usepackage[shortlabels]{enumitem}
\usepackage{bbm}

\hypersetup{breaklinks=true}

\title{Constant rank operators in Korn-Maxwell-Sobolev inequalities}
\author[P. Lewintan]{Peter Lewintan}
\address[Peter Lewintan]{Institute for Analysis, Karlsruhe Institute of Technology, Englerstr. 2, 76131 Karlsruhe, Germany}
\email{peter.lewintan@kit.edu}
\author[P. Stephan]{Paul Stephan}
\address[Paul Stephan]{Department of Mathematics and Statistics, University of Konstanz, Universit\"{a}tsstrasse 10, 78464 Konstanz, Germany.}
\email{paul.stephan@uni-konstanz.de}

\usepackage{libertine}

\newcommand{\B}{\mathbb{B}}
\newcommand{\A}{\mathscr{A}}

\usepackage{tikz}
\usepackage{scalerel}

\numberwithin{equation}{section}

\newcommand{\bbone}{\text{\usefont{U}{bbold}{m}{n}1}}
\MakeRobust{\bbone}
\newcommand{\R}{\mathbb R}
\newcommand{\N}{\mathbb N}

\newcommand{\D}{\mathrm{D}}
\DeclareMathOperator{\Curl}{Curl}
\DeclareMathOperator{\Div}{Div}

\newcommand{\dev}{\operatorname{dev}}
\newcommand{\sym}{\operatorname{sym}}

\DeclareMathOperator{\tr}{tr}

\newcommand{\lebe}{\operatorname{L}}
\newcommand{\sobo}{\operatorname{W}}
\newcommand{\hold}{\operatorname{C}}

\newcommand{\norm}[1]{\left\lVert#1\right\rVert}

\newcommand{\abs}[1]{\lvert#1\rvert}

\theoremstyle{plain}
\newtheorem{theorem}{Theorem}[section]
\newtheorem{lemma}[theorem]{Lemma}

\newtheorem{corollary}[theorem]{Corollary}

\theoremstyle{remark}
\newtheorem{remark}[theorem]{Remark}
\newtheorem{example}[theorem]{Example}

\begin{document}

        \begin{abstract}
		We focus on Korn-Maxwell-Sobolev inequalities for operators of reduced constant rank. These inequalities take the form  
		\[
		\norm{P - \Pi_{\B} \Pi_{\ker\A} P}_{\dot{\sobo}^{k-1, p^*}(\R^n)} \le c \, (\norm{\A[P]}_{\dot{\sobo}^{k-1, p^*}(\R^n)} + \norm{\B P}_{\lebe^p(\R^n)})
		\]  
		for all \( P \in \hold_c^\infty(\R^n; V) \), where \( V \) is a finite-dimensional vector space, \( \A \) is a linear mapping, and \( \mathbb{B} \) is a constant coefficient homogeneous differential operator of order \( k \). In particular, we can treat the combination $(p,\A,\B,k)=(1,\tr,\Curl,1)$. Our results generalize the techniques from \cite{GLN2, GLVS}, which exclusively dealt with reduced elliptic operators. In contrast to the reduced ellipticity case, however, the reduced constant rank case necessitates to introduce a correction, namely the projection $\Pi_\mathbb{B}$
		on the left-hand side of the inequality.
	\end{abstract}

	% \begin{abstract} In the recent contributions \cite{GLN2, GLVS} the complete classification of the operators that might appear in Korn-Maxwell-Sobolev inequalities was given. In particular (reduced) ellipticity should be present. In this short note, we focus on a more general class of differential operators and show how these inequalities have to be corrected in the absence of (reduced) ellipticity. This is the case, e.g., for the right-hand side of the following form:
		% 	\begin{align*}
			% 		\norm{\tr P}_{\lebe^{\frac{n}{n-1}}(\R^n)} + \norm{\Curl P}_{\lebe^1(\R^n)}  \qquad \text{with } P\in\hold^\infty_c(\R^n;\R^{n\times n}).
			% 	\end{align*}
		% This example is discussed in detail below.
		% \end{abstract}
	\date\today
	\maketitle

	\section{Introduction}
	Korn's inequalities provide the coercivity of the bilinear forms appearing in linear elasticity or fluid mechanics and in their most prominent version it reads (for $n\ge2$):
	\begin{align}\label{eq:Korn0}
		\norm{\D u}_{\lebe^2(\R^n)}\le c\,\norm{\sym \D u}_{\lebe^2´(\R^n)} \quad\forall\ u\in\hold^\infty_c(\R^n;\R^n).
	\end{align}
	For further relaxation of the models, one might ask for a better understanding of the possible part maps that may appear on the right-hand side of such inequalities. Indeed, a complete classification of admissible linear maps $\A:\R^{n\times n}\to V$ follows from the fundamental paper by Calderón and Zygmund \cite{CZ} for $1<q<\infty$:
	\begin{align*}
		\norm{\D u}_{\lebe^q(\R^n)}\le c\,\norm{\A[\D u]}_{\lebe^q(\R^n)} \quad\Leftrightarrow \quad \A[\D\cdot] \text{ elliptic},
	\end{align*}
	where a $k$-th order linear homogeneous constant coefficients differential operator $$\mathbb{A}=\sum_{\abs{\alpha}=k}\mathbb{A}_\alpha\partial^\alpha,$$ with linear maps $\mathbb{A}_\alpha$ between two finite dimensional spaces $E$ and $F$, is called \emph{elliptic} if and only if the corresponding symbol maps $$\mathbb{A}[\xi]=\sum_{\abs{\alpha}=k}\mathbb{A}_\alpha\xi^\alpha: E\to F$$ are all injective for all $\xi\in\R^n\backslash\{0\}$. Note that the classical Korn's inequalities do not hold in the boundary case $q=1$ due to Ornstein's non-inequality, cf. \cite{Conti_New, Diening_Sharp_2024, Ornstein}.
	\medskip
	
	Here, we focus on an extension in a different direction, e.g. in the strain gradient theory, a macroscopic theory of plasticity, the displacement gradient is assumed to decompose into an elastic part $e$ and a plastic part $P$, both non-symmetric and incompatible. Then dislocations (e.g. coming from crystal defects) are modeled by controlling the curl of the plastic distortion which also enters the energy. To prove coercivity in case of incompatible fields (i.e. not gradients), Korn's inequality has to observe the absence of curl-freeness and, indeed, it holds, cf. \cite{ContiGarroni,GLP,Neffunifying,NeffPlastic,NeffPaulyWitsch}:
	\begin{align}\label{eq:KMS0}
		\norm{P}_{\lebe^q(\R^n)}\le c\,(\norm{\sym P}_{\lebe^q(\R^n)}+\norm{\Curl P}_{\lebe^p(\R^n)}) \quad \forall P\in\hold^\infty_c(\R^n;\R^{n\times n}),
	\end{align}
	where $\Curl$ denotes the matrix curl, i.e. the row-wise application of the vector curl, and,  in terms of scaling the natural relation between the integrabilities $q$ and $p$ is $q=p^*=\frac{np}{n-p}$ being the Sobolev exponent of $p\in[1,n)$. Testing \eqref{eq:KMS0} with gradients, we recover \eqref{eq:Korn0}. Thus, the reduction to the classical Korn's inequality, the presence of the Curl and the Sobolev conjugate exponent motivates the terminology of Korn-Maxwell-Sobolev introduced in \cite{GmSp} for such type of inequalities. In the most general case, such inequalities read as
	\begin{align}\label{eq:KMS-I}
		\norm{P}_{\lebe^{q}(\R^n)}\le c\left(\norm{\A[P]}_{\lebe^{q}(\R^n)}+\norm{\B P}_{\lebe^p(\R^n)}\right), \qquad  P\in\hold^\infty_c(\R^n;V), \tag{KMS-I}
	\end{align}
	where $V,\widetilde V, W$ are finite dimensional vector spaces, $\A: V\to\widetilde V$ is a linear map and $\B$ is a linear homogeneous first order differential operator on $\R^n$ from $V$ to $W$ with constant coefficients. In the sequel, we shall refer to \eqref{eq:KMS-I} as \emph{Korn-Maxwell-Sobolev inequalities of first kind}.  Variants of such inequalities have already been used to obtain error estimates in mixed finite elements methods, cf. \cite{Arnold}, where the particular combination $\A=\dev$ and $\B=\Div$ was considered. Moreover, \eqref{eq:KMS-I} with $\A=\sym$ and $\B=\Curl$ provide the coercivity of the appearing bilinear forms in extended continuum models, like the relaxed micromorphic model \cite{Neffunifying, Neff}. There exists a vast literature on various constellations in the case $p>1$, see e.g. the references in \cite{BauerNeffPaulyStarke,GLN2,LMN}, and a complete characterization of the interplay between the part map $\A$ and the differential operator $\B$ in Korn-Maxwell-Sobolev inequalities has been given in \cite{GLN2}. On the contrary, only few results, namely \cite{ContiGarroni,GLP,GLN1,GmSp}, address the limiting case $p=1$ and only in the particular situation of the differential operator $\B=\Curl$ being the matrix curl. This gap has been closed in the recent paper \cite{GLVS} where again a complete classification was obtained. The results in \cite{GLN2,GLVS} are optimal in the following sense:
	\begin{itemize}
		\item if $1<p<n$ then:
		\begin{align*}
			\eqref{eq:KMS-I} \quad \xLeftrightarrow{\text{\cite{GLN2}}} \quad \ker\A\cap\bigcup_{\xi\in\R^n\backslash\{0\}}\ker\B[\xi]=\{0\},
		\end{align*}
		\item for $p=1$:
		\begin{align*}
			\eqref{eq:KMS-I} \quad \xLeftrightarrow{\text{\cite{GLVS}}} \quad  \left(\ \ker\A\cap\bigcup_{\xi\in\R^n\backslash\{0\}}\ker\B[\xi]=\{0\} \quad \wedge \ \bigcap_{\xi\in\R^n\backslash\{0\}}\B[\xi](\ker\A)=\{0\}\right) .
		\end{align*}
	\end{itemize}
	Note that the first algebraic condition expresses the injectivity of the restricted symbol maps $\B[\xi]:\ker\A\to W$, and the additional condition in the borderline case $p=1$ comes from the extra condition of \textit{cancellation} in limiting Sobolev inequalities, cf. \cite{VS}, where this new terminology was introduced. Recall, that a differential operator $\mathbb{A}$ as above is called \textit{cancelling} if and only if the intersection of the images of all symbol maps is trivial:
	\begin{align*}
		\bigcap_{\xi\in\R^n\backslash\{0\}}\operatorname{im}\mathbb{A}[\xi]=\bigcap_{\xi\in\R^n\backslash\{0\}}\mathbb{A}[\xi](E)=\{0\}.
	\end{align*}
	A strengthened condition is the one of $\mathbb{C}$-ellipticity, and we refer to \cite{GmRaVa1} for the  connections of the single notions for differential operators. Given the above classifications, for validity of \eqref{eq:KMS-I} it only remains to check the mentioned algebraic conditions. Indeed, the combinations $(\A,\B)=(\dev,\Div)$ or $(\A,\B)=(\sym,\Curl)$ are admissible in \eqref{eq:KMS-I}, cf. \cite{GLN2,GLVS}, whereas with the constellation $(\A,\B)=(\tr,\Curl)$ the algebraic condition of reduced ellipticity is violated.
	
	\section{Constant rank operators}
	The motivation of the present note is to understand in which sense the latter combination $(\A,\B)=(\tr,\Curl)$ can appear on the right-hand side of \eqref{eq:KMS-I}, see Example \ref{ex:trCurl} below. This brings us to more general differential operators, namely \textit{constant rank} operators. A differential operator $\mathbb{A}$ as above is said to have a \textit{constant rank} if and only if there exists $r\in\N_0$ such that
	\begin{align*}
		\operatorname{rank} \mathbb{A}[\xi]=r \qquad \forall\ \xi\in\R^n\backslash\{0\}.
	\end{align*}
	In particular, all elliptic operators have constant rank, but also the curl and the divergence (see \cite[Rem. 3.3]{FM}). It is precisely this condition of constant rank that was used in \cite{SW} to prove coercivity estimates for non-elliptic systems, and it was used in the context of compensated compactness by Murat \cite{Murat}. However, we know from the Calderón-Zygmund theory that ellipticity is both necessary and sufficient in coercivity inequalities, so that for constant rank operators we need a correction that captures non-ellipticity. This is observed by the map $\Pi_{\mathbb{A}}$ given by
	\begin{align*}
		\Pi_{\mathbb{A}}u\coloneqq\mathscr{F}^{-1}\left(\Pi_{\ker\mathbb{A}[\xi]}[\mathscr{F}(u(x))]\right) \qquad \text{for } u\in\hold^\infty_c(\R^n;E),
	\end{align*}
	where $\mathscr{F}$ denotes the Fourier transform and $\Pi_{\ker\mathbb{A}[\xi]}$ the projection on $\ker\mathbb{A}[\xi]$. Then constant rank operators are characterized by the following lemma:
	\begin{lemma}[\cite{FM,GR,SW}]\label{lem:classCR}
		Let $n\ge2$, $1<q<\infty$, $k\in\N$, $E$ and $F$ be finite dimensional spaces and $\mathbb{A}$ be a $k$-th order linear homogeneous constant coefficients differential operator on $\R^n$ from $E$ to $F$. Then $\mathbb{A}$ has constant rank if and only if there exists a constant $c=c(q,\mathbb{A})>0$ such that
		\begin{align*}
			\norm{u-\Pi_{\mathbb{A}}u}_{\lebe^q(\R^n)}\le c\, \norm{\mathbb{A}u}_{\dot{\sobo}^{-k,q}(\R^n)} \quad \forall\ u\in\hold^\infty_c(\R^n;E).
		\end{align*}
	\end{lemma}
	The sufficiency of the constant rank condition in Lemma \ref{lem:classCR} is contained in \cite{SW,FM} and the necessity was shown in \cite{GR}.
	\begin{example} \label{ex:curl}
		For $E=\R^n$ and $\mathbb{A}=\operatorname{curl}$, the projection $\Pi_{\mathbb{A}}$ maps to the $\operatorname{curl}$-free part from the Helmholtz decomposition of $u$:
		\begin{align*}
			\Pi_{\operatorname{curl}} u (x)= u_{\operatorname{curl}}(x)=\frac{1}{n\omega_n}\int_{\R^n}\frac{x-y}{\abs{x-y}^n}\operatorname{div} u(y)\,\mathrm{d} y.
		\end{align*}
	\end{example}
	
	\begin{remark} Since for constant rank operators it holds
		\begin{align}\label{eq:projellipt}
			\Pi_{\mathbb{A}}= 0 \quad \Leftrightarrow \quad \mathbb{A} \text{ is elliptic},
		\end{align}
		we directly recover from Lemma \ref{lem:classCR} Calderón-Zygmund's classification of differential operators that might appear in classical Korn inequalities of the first type \cite{CZ}:
	\end{remark}
	\begin{corollary}
		Let $n\ge2$, $1<q<\infty$, $k\in\N$, $E$ and $F$ be finite dimensional spaces and $\mathbb{A}$ be a $k$-th order linear homogeneous constant coefficients differential operator on $\R^n$ from $E$ to $F$. Then $\mathbb{A}$ is elliptic if and only if there exists a constant $c=c(q,\mathbb{A})>0$ such that
		\begin{align*}
			\norm{u}_{\lebe^q(\R^n)}\le c\, \norm{\mathbb{A}u}_{\dot{\sobo}^{-k,q}(\R^n)} \quad \forall\ u\in\hold^\infty_c(\R^n;E).
		\end{align*}
	\end{corollary}\medskip
	
	Now, we are prepared to catch up with the statement of our main theorem:
	
	\begin{theorem}\label{thm:main}
		Let $n\ge2$, $k\in\N$, $V$, $\widetilde V$ and $W$ be finite dimensional spaces, $\A: V\to\widetilde V$ be linear and $\B$ be a $k$-th order linear homogeneous constant coefficients differential operator on $\R^n$ from $V$ to $W$.
		\begin{enumerate}[i)]
			\item If $\B$ has a \emph{reduced constant rank relative to $\A$}, meaning that \label{item:p>1}
			\begin{align*}
				\B[\xi]\big|_{\ker\A} : \ker\A\to W \quad  \text{has constant rank for all $\xi\in\R^n\backslash\{0\}$} 
			\end{align*}
			then for all $1<p<n$ there exists a constant $c=c(\A,\B,p)>0$ such that
			\begin{align*}
				\norm{P-\Pi_{\B}\Pi_{\ker\A}P}_{\dot{\sobo}^{k-1,p^*}(\R^n)} \le c\, (\norm{\A[P]}_{\dot{\sobo}^{k-1,p^*}(\R^n)}+\norm{\B P}_{\lebe^p(\R^n)}).
			\end{align*}
			\item If \ $\B$ has a \emph{reduced constant rank relative to $\A$} and $\B$ is \emph{reduced cancelling relative to $\A$}, meaning that \label{item:p=1}
			\begin{align*}
				\B[\xi]\big|_{\ker\A} : \ker\A\to W \  \text{has constant rank for all $\xi\in\R^n\backslash\{0\}$} \end{align*} and \begin{align*} \ \bigcap_{\xi\in\R^n\backslash\{0\}}\B[\xi](\ker\A)=\{0\}
			\end{align*}
			then there exists a constant $c=c(\A,\B)>0$ such that
			\begin{align*}
				\norm{P-\Pi_{\B}\Pi_{\ker\A}P}_{\dot{\sobo}^{k-1,1^*}(\R^n)} \le c\, (\norm{\A[P]}_{\dot{\sobo}^{k-1,1^*}(\R^n)}+\norm{\B P}_{\lebe^1(\R^n)}).
			\end{align*}
		\end{enumerate}
	\end{theorem}
	\begin{example}\label{ex:trCurl}
		$\B=\Curl$ is (globally) cancelling and has (globally) constant rank, thus, $\A$ can be any linear map. For $\A=\tr$ we have $\Pi_{\ker\A}=\dev$ denoting the deviatoric (i.e., trace-free) part, i.e. $\dev P = P-\frac{\tr P}{n}\bbone_n$, so that in view of Example \ref{ex:curl} we have
		\begin{align*}
			\Pi_{\B}\Pi_{\ker\A}P(x)=\frac{1}{n\omega_n}\int_{\R^n}\frac{x-y}{\abs{x-y}^n}\otimes\Div\dev P(y)\,\mathrm{d} y,
		\end{align*}
		where $\Div$ is the matrix divergence, i.e. the row-wise application of the vector divergence, and we obtain:
		\begin{align*}
			&\left(\int_{\R^n}\left\lvert P(x)-\frac{1}{n\omega_n}\int_{\R^n}\frac{x-y}{\abs{x-y}^n}\otimes\Div\dev P(y)\,\mathrm{d} y \right\rvert^{\frac{n}{n-1}} \,\mathrm{d}x\right)^{\frac{n-1}{n}}\\
			&\hspace{9em}\le c\,\left(
			\left(\int_{\R^n}\left\lvert \tr P(x)\right\rvert^{\frac{n}{n-1}} \,\mathrm{d}x\right)^{\frac{n-1}{n}}+\int_{\R^n}\abs{\Curl P(x)}\,\mathrm{d} x\right).
		\end{align*}
	\end{example}
	In view of \eqref{eq:projellipt} we also recover from our main Theorem \ref{thm:main} the sufficiency of ellipticity in Korn-Maxwell-Sobolev inequalities known from \cite{GLN2} and \cite{GLVS}:
	\begin{corollary}
		Let $n\ge2$, $k\in\N$, $V$, $\widetilde V$ and $W$ be finite dimensional spaces, $\A: V\to\widetilde V$ be linear and $\B$ be a $k$-th order linear homogeneous constant coefficients differential operator on $\R^n$ from $V$ to $W$.
		\begin{enumerate}[i)]
			\item If $\B$ is \emph{reduced elliptic relative to $\A$}, meaning that
			\begin{align*}
				\B[\xi]\big|_{\ker\A} : \ker\A\to W \quad  \text{has maximal rank for all $\xi\in\R^n\backslash\{0\}$} 
			\end{align*}
			then for all $1<p<n$ there exists a constant $c=c(\A,\B,p)>0$ such that
			\begin{align*}
				\norm{P}_{\dot{\sobo}^{k-1,p^*}(\R^n)} \le c\, (\norm{\A[P]}_{\dot{\sobo}^{k-1,p^*}(\R^n)}+\norm{\B P}_{\lebe^p(\R^n)}).
			\end{align*}
			\item If \ $\B$ is \emph{reduced elliptic relative to $\A$} and $\B$ is \emph{reduced cancelling relative to $\A$}, meaning that
			\begin{align*}
				\B[\xi]\big|_{\ker\A} : \ker\A\to W \  \text{has maximal rank for all $\xi\in\R^n\backslash\{0\}$} \end{align*} and \begin{align*} \ \bigcap_{\xi\in\R^n\backslash\{0\}}\B[\xi](\ker\A)=\{0\},
			\end{align*}
			then there exists a constant $c=c(\A,\B)>0$ such that
			\begin{align*}
				\norm{P}_{\dot{\sobo}^{k-1,1^*}(\R^n)} \le c\, (\norm{\A[P]}_{\dot{\sobo}^{k-1,1^*}(\R^n)}+\norm{\B P}_{\lebe^1(\R^n)}).
			\end{align*}
		\end{enumerate}
	\end{corollary}

	\section{Proof of the main theorem}
	The proof of the main Theorem \ref{thm:main} follows from the next lemma:
	\begin{lemma}\label{lem:negSob}
		Let $n\ge2$, $1< q<\infty$, $k\in\N$, $V$, $\widetilde V$ and $W$ be finite dimensional spaces, $\A: V\to\widetilde V$ be linear and $\B$ be a $k$-th order linear homogeneous constant coefficients differential operator on $\R^n$ from $V$ to $W$. Then the following are equivalent:
		\begin{enumerate}[(a)]
			\item There exists a constant $c=c(\A,\B,q)>0$ such that \label{item:auxineq}
			\begin{align}\label{eq:wishLq}
				\norm{P-\Pi_{\B}\Pi_{\ker\A}P}_{\lebe^q(\R^n)} \le c\, (\norm{\A[P]}_{\lebe^q(\R^n)}+\norm{\B P}_{\dot{\sobo}^{-k,q}(\R^n)}), 
			\end{align}
			\item $\B$ has a \emph{reduced constant rank relative to $\A$}, meaning that \label{item:condition}
			\begin{align*}
				\B[\xi]\big|_{\ker\A} : \ker\A\to W \quad  \text{has constant rank for all $\xi\in\R^n\backslash\{0\}$}.
			\end{align*}
		\end{enumerate}
	\end{lemma}
	
    \begin{proof} ``\ref{item:condition} $\Rightarrow$ \ref{item:auxineq}'':  To establish the sufficiency of the constant rank condition, we start with the pointwise decomposition $P=\Pi_{\ker\A}P+\Pi_{(\ker\A)^\perp}P$. Note first that $\A\circ\Pi_{(\ker\A)^\perp}$ is injective, otherwise we would find $u,v\in(\ker\A)^\perp$ such that $\A u =\A v$, meaning that $u-v\in\ker\A$. Hence, there exists a constant $c=c(\A)>0$ such that we can even estimate pointwise:
		\begin{align}\label{eq:goodguy}
			\abs{\Pi_{(\ker\A)^\perp}P}\le c\, \abs{\A[\Pi_{(\ker\A)^\perp}P]} = c\, \abs{\A[P]}.
		\end{align}
		For the remaining part we apply Lemma \ref{lem:classCR} with $E=\ker\A$, $u=\Pi_{\ker A}P$ and the $k$-th order differential operator $\B$:
		\begin{align}
			\norm{\Pi_{\ker A}P-\Pi_\B\Pi_{\ker A}P}_{\lebe^q(\R^n)} & \le c\norm{\B\Pi_{\ker A}P}_{\dot{\sobo}^{-k,q}(\R^n)}\notag\\
			&\le c \norm{\B P}_{\dot{\sobo}^{-k,q}(\R^n)} + c\norm{\B\Pi_{(\ker\A)^\perp}P}_{\dot{\sobo}^{-k,q}(\R^n)}\notag\\
			& \overset{\mathclap{\B\ k\text{th ord.}}}{\le}\quad  c \norm{\B P}_{\dot{\sobo}^{-k,q}(\R^n)} + c\norm{\Pi_{(\ker\A)^\perp}P}_{\lebe^{q}(\R^n)}\notag\\
			&\overset{\mathclap{\eqref{eq:goodguy}}}{\le} c \norm{\B P}_{\dot{\sobo}^{-k,q}(\R^n)} + c\norm{\A[P]}_{\lebe^{q}(\R^n)} \label{eq:badguy}
		\end{align}
		and a combination of \eqref{eq:goodguy} and \eqref{eq:badguy} gives the desired inequality \eqref{eq:wishLq}.
		
		For the converse implication ``\ref{item:auxineq} $\Rightarrow$ \ref{item:condition}'' test \eqref{eq:wishLq} with functions $P\in\hold^\infty_c(\R^n;\ker\A)$ and obtain
		\begin{align*}
			\norm{P-\Pi_{\B}P}_{\lebe^q(\R^n)}\le c \norm{\B P}_{\dot{\sobo}^{-k,q}(\R^n)}
		\end{align*}
		so that the necessity of the constant rank condition is a consequence of Lemma \ref{lem:classCR}. 
	\end{proof}
	
	Now we can complete the 
	
	\begin{proof}[Proof of Theorem \ref{thm:main}.] Ad \ref{item:p>1} Let $1<p<n$. For $q=p^*=\frac{np}{n-p}$ we obtain from Lemma \ref{lem:negSob}:
		\begin{align*}
			\norm{P-\Pi_{\B}\Pi_{\ker\A}P}_{\dot{\sobo}^{k-1,p^*}(\R^n)} &\le c\, (\norm{\A[P]}_{\dot{\sobo}^{k-1,p^*}(\R^n)}+\norm{\B P}_{\dot{\sobo}^{-1,p^*}(\R^n)}) \\
			& \le c\, (\norm{\A[P]}_{\dot{\sobo}^{k-1,p^*}(\R^n)}+\norm{\B P}_{\lebe^p(\R^n)}),
		\end{align*}
		by classical Sobolev embedding $\lebe^p(\R^n)\hookrightarrow \dot{\sobo}^{-1,p^*}(\R^n)$.\medskip
		
		Ad \ref{item:p=1} Note, that the last step fails for $p=1$. However, we can proceed like in \cite{GLVS}. The crucial idea is, that constant rank operators allow for a (ellipticity) complex, cf. \cite[Thm 1]{Raita}. Thus, since $\B:\hold^\infty_c(\R^n;\ker\A)\to\hold^\infty_c(\R^n;W)$ has constant rank we find a linear homogeneous constant coefficients differential operator $\mathbb{L}$ of order $\ell$ such that
		\begin{align}\label{eq:complex}
			\ker \mathbb{L}[\xi]=\B[\xi](\ker\A) \qquad \forall\ \xi\in\R^n\backslash\{0\}.
		\end{align}
		By the reduced cancellation of $\B$ relative to $\A$ we deduce the cocancellation of $\mathbb{L}$:
		\begin{align*}
			\bigcap_{\xi\in\R^n\backslash\{0\}} \ker \mathbb{L}[\xi]=\bigcap_{\xi\in\R^n\backslash\{0\}}\B[\xi](\ker\A)=\{0\}.
		\end{align*}
		Hence, we can apply Van Schaftingen's strong Bourgain-Brezis estimate for cocancelling differential operators \cite[Thm 9.2]{VS}:
		\begin{align*}
			\norm{f}_{\dot{\sobo}^{-1,1^*}(\R^n)}\le c\,(\norm{\mathbb{L}f}_{\dot{\sobo}^{-1-\ell,1^*}(\R^n)}+\norm{f}_{\lebe^1(\R^n)}) \qquad \forall\ f\in\hold^\infty_c(\R^n;W)
		\end{align*}
		with the particular choice $f=\B P$:
		\begin{align*}
			\norm{\B P}_{\dot{\sobo}^{-1,1^*}(\R^n)}&\le c\,(\norm{\mathbb{L}\B P}_{\dot{\sobo}^{-1-\ell,1^*}(\R^n)}+\norm{\B P}_{\lebe^1(\R^n)})\\
			& \le c\,(\lVert\underbrace{\mathbb{L}\B \Pi_{\ker\A}P}_{=0} + \mathbb{L}\B \Pi_{(\ker\A)^{\perp}}P\rVert_{\dot{\sobo}^{-1-\ell,1^*}(\R^n)}+\norm{\B P}_{\lebe^1(\R^n)})\\
			&\overset{\eqref{eq:complex}}{\le} c\,(\norm{\mathbb{L}\B \Pi_{(\ker\A)^{\perp}}P}_{\dot{\sobo}^{-1-\ell,1^*}(\R^n)}+\norm{\B P}_{\lebe^1(\R^n)})\\
			&\overset{\mathclap{\mathbb{L}\ \ell\text{th ord.}}}{\le}\quad c\,(\norm{\B \Pi_{(\ker\A)^{\perp}}P}_{\dot{\sobo}^{-1,1^*}(\R^n)}+\norm{\B P}_{\lebe^1(\R^n)})\\
			&\overset{\mathclap{\B\ k\text{th ord.}}}{\le}\quad c\,(\norm{\Pi_{(\ker\A)^{\perp}}P}_{\dot{\sobo}^{k-1,1^*}(\R^n)}+\norm{\B P}_{\lebe^1(\R^n)})\\
			&\overset{\eqref{eq:goodguy}}{\le} c\,(\norm{\A[P]}_{\dot{\sobo}^{k-1,1^*}(\R^n)}+\norm{\B P}_{\lebe^1(\R^n)})
		\end{align*}
		which allows us to conclude like in the proof of \ref{item:p>1}. 
	\end{proof}
	
	\begin{remark}
		Note that if $\ker\A\subseteq\ker\B[\xi]$ then $\Pi_{\B}\Pi_{\ker \A}=\Pi_{\ker\A}$ and the KMS inequality with the presence of reduced constant rank property reads:
		\begin{align*}
			\norm{\Pi_{(\ker\A)^{\perp}}[P]}_{\dot{\sobo}^{k-1,p^*}(\R^n)}\le c\,(\norm{\A[P]}_{\dot{\sobo}^{k-1,p^*}(\R^n)}+\norm{\B P}_{\lebe^p(\R^n)})
		\end{align*}
		which is obvious in view of \eqref{eq:goodguy}. Indeed, such a case appears, e.g., for the combination  $(\A,\B,n)=(\dev,\sym\Curl,3)$ where $\Pi_{\ker\A}P=\frac{\tr P}{3}\mathbbm{1}_3$ and $\dev P = P -\frac{\tr P}{3}\mathbbm{1}_3=\Pi_{(\ker\A)^\perp}P$.
		
	\end{remark}
	
	\section{Necessity}
	In terms of applications, the sufficiency of our results is already a valuable tool to check the validity of the desired coercivity estimates. However, for a sharp result, the necessity is also obligatory. Thus, in Lemma \ref{lem:negSob} we establish the equivalence between the reduced constant rank condition and the validity of an $\lebe^q$-estimate involving negative Sobolev norms. This suggests that the reduced constant rank condition in Theorem \ref{thm:main} is not only sufficient but also, at least in related settings, necessary. In particular, this would rely on the following classification:
	\begin{align}   
	 \norm{u-\Pi_{\mathbb{A}}u}_{\dot{\sobo}^{k-1,p^*}(\R^n)}\le c \norm{\mathbb{A} u}_{\lebe^p(\R^n)} \quad \Leftrightarrow\quad \mathbb{A} &\text{ has constant rank} \label{eq:DESIRE} \\ &\text{ and, in case $p=1$, is, in addition, cancelling.}\notag
	\end{align}
The sufficiency here follows from Lemma \ref{lem:classCR} and the Sobolev embedding in case $p\in(1,n)$ and, in the borderline case $p=1$, from the limiting Sobolev inequality given by Van Schaftingen \cite{VS}:

\begin{itemize}
 \item If $p\in(1,n)$ we have by Lemma \ref{lem:classCR} for a constant rank operator $\mathbb{A}$:
 \[
  \norm{u-\Pi_{\mathbb{A}}u}_{\dot{\sobo}^{k-1,p^*}(\R^n)}\le c \norm{\mathbb{A} u}_{\dot{\sobo}^{-1,p^*}(\R^n)} \le c\norm{\mathbb{A} u}_{\lebe^p(\R^n)} .
 \]
\item If $p=1$, we also start by Lemma \ref{lem:classCR}, but we cannot generally estimate the $\dot{\sobo}^{-1,1^*}$-norm by the $\lebe^1$-norm. However, this last step can be performed if $\mathbb{A}$ is cancelling. Indeed, it follows from \cite[Thm 1.4]{VS} that for a cancelling operator $\mathbb{A}$ it holds
\[
  \norm{\mathbb{A} u}_{\dot{\sobo}^{-1,1^*}(\R^n)} \le c\norm{\mathbb{A} u}_{\lebe^1(\R^n)},
\]
which then completes the sufficiency in \eqref{eq:DESIRE} also in the borderline case $p=1$.
\end{itemize}
However, the necessity of the constant rank condition in \eqref{eq:DESIRE} remains an open question and will be addressed elsewhere. For a comparable discussion involving elliptic differential operators, we refer the interested reader to \cite[Sec. 5]{VS}.

	{\footnotesize
		\subsection*{Conflict of interest} The authors declare that they have no conflict of interest.
		\vspace{-2.5ex}
		\subsection*{Data availability statement} Data sharing not applicable to this article as no datasets were generated or analyzed.
	}

\end{document}